\documentclass[a4paper,11pt]{article}

\usepackage[T1]{fontenc}
\usepackage{hyperref}
\usepackage{float}
\usepackage{amssymb,amscd,amsmath,amsfonts,amsthm}
\usepackage{enumerate}
\usepackage{multicol}
\usepackage{tabularx,environ,array}
\newcolumntype{C}{>{\centering\arraybackslash}X}
\usepackage{tikz}
\usepackage[numbers,sort]{natbib} %Sorting references
\usepackage{authblk} %Multiple authors
\usepackage{natbib}

\newtheorem{theorem}{Theorem}

\newtheorem{corollary}[theorem]{Corollary}

\newtheorem{lemma}[theorem]{Lemma}

\newtheorem{observation}[theorem]{Observation}
\newtheorem{problem}[theorem]{Problem}
\newtheorem{proposition}[theorem]{Proposition}
\newtheorem{remark}[theorem]{Remark}

\newcommand{\Kneser}[2]{\ensuremath{K}({#1,#2})}

\newcommand{\N}{\mathbb{N}}

\newcommand{\ktuple}[1]{\gamma_{\times #1}}
\newcommand{\ktuplet}[1]{\gamma_{t\times #1}}
\newcommand{\kdom}[1]{\gamma_{#1}}

\newcommand{\packing}{\rho_{2}}

\title{$k$-Domination invariants on Kneser graphs}
\date{\today}

\author[1,2]{Bo\v{s}tjan Bre\v{s}ar}
\author[3,4]{Mar\'ia Gracia Cornet}
\author[1,2]{Tanja Dravec}
\author[5]{Michael A. Henning}

\affil[1]{\footnotesize Faculty of Natural Sciences and Mathematics, University of Maribor, Slovenia}
\affil[2]{\footnotesize Institute of Mathematics, Physics and Mechanics, Ljubljana, Slovenia}
\affil[3]{\footnotesize Depto. de Matem\'atica, FCEIA, Universidad Nacional de Rosario, Argentina}
\affil[4]{\footnotesize Consejo Nacional de Investigaciones Científicas y Técnicas, Argentina}
\affil[5]{\footnotesize Department of Mathematics and Applied Mathematics, University of Johannesburg, South Africa}

\begin{document}

\maketitle

\begin{abstract}
In this follow-up to [M.G.~Cornet, P.~Torres, arXiv:2308.15603], where the $k$-tuple domination number and the 2-packing number in Kneser graphs $K(n,r)$ were studied, we are concerned with two variations, the $k$-domination number, $\kdom{k}(K(n,r))$, and the $k$-tuple total domination number, $\ktuplet{k}(K(n,r))$, of $K(n,r)$. For both invariants we prove monotonicity results by showing that $\kdom{k}(K(n,r))\ge \kdom{k}(K(n+1,r))$  holds for any $n\ge 2(k+r)$, and $\ktuplet{k}(K(n,r))\ge \ktuplet{k}(K(n+1,r))$ holds for any $n\ge 2r+1$. We prove that $\kdom{k}(K(n,r))=\ktuplet{k}(K(n,r))=k+r$ when $n\geq r(k+r)$, and that in this case every $\kdom{k}$-set and $\ktuplet{k}$-set is a clique, while $\kdom{k}(r(k+r)-1,r)=\ktuplet{k}(r(k+r)-1,r)=k+r+1$, for any $k\ge 2$.  
Concerning the 2-packing number, $\rho_2(K(n,r))$, of $K(n,r)$, we prove the exact values of $\rho_2(K(3r-3,r))$ when $r\ge 10$, and give sufficient conditions for $\rho_2(K(n,r))$ to be equal to some small values by imposing bounds on $r$ with respect to $n$. We also prove a version of monotonicity for the $2$-packing number of Kneser graphs.   
\end{abstract}


\medskip


\section{Introduction}
\label{sec: intro}

%\section{Definitions}
%\label{sec: def}
Letting $n\geq 2r$ the {\em Kneser graph} $\Kneser{n}{r}$ has the $r$-subsets of an $n$-set as its vertices and two vertices are adjacent in $\Kneser{n}{r}$ if the corresponding sets are disjoint.
The interest for Kneser graphs goes back to 1960s and 1970s when two classical theorems concerning their independence and chromatic number were proved~\cite{erdos1961,lovasz1978}. Many other graph invariants have been investigated in Kneser graphs, which makes them one of the most intensively studied classes of graphs. In particular, several authors have considered the domination number of Kneser graphs (see the most recent paper~\cite{ostergaard2014bounds}), but only lower and upper bounds have been found in general, while the exact values of $\gamma(\Kneser{n}{r})$ were found only when $n$ is sufficiently large with respect to~$r$. In three recent papers, the authors have considered Kneser graphs in relation with some variations of domination, namely, Grundy domination~\cite{bkt-2019}, Roman domination~\cite{zec2023} and $k$-tuple domination~\cite{cornet2023k}. This paper is a follow-up of the latter, where we extend our investigation from $k$-tuple domination to two similar variations, which we next present.

A \emph{dominating set} in a graph $G$ is a set $S$ of vertices of $G$ such that every vertex outside $S$ is adjacent to at least one vertex in $S$. The \emph{domination number} of $G$, denoted by $\gamma(G)$, is the minimum cardinality of a dominating set in $G$. A thorough treatise on dominating sets can be found in the so-called ``domination books''~\cite{HaHeHe-20,HaHeHe-21,HaHeHe-23,henningyeo-13}.

Let $G$ be a graph and $k\in\N$. A set $D\subseteq V(G)$ is a \textit{$k$-dominating set} of $G$ if every vertex $u\in V(G)\setminus D$ has at least $k$ neighbors in $D$ (or, equivalently, if $|N_G(u)\cap D|\geq k$ for each $u\in V(G)\setminus D$ where $N_G(u)$ denotes the open neighborhood of the vertex~$u$ in $G$). The \textit{$k$-domination number} of $G$ is the minimum cardinality of a $k$-dominating set of $G$, and is denoted by $\kdom{k}(G)$. A {\em $\kdom{k}$-set} of $G$ is a $k$-dominating set of cardinality $\kdom{k}(G)$. See~\cite{fink1985n1,fink1985n2,chellali2012k,hansberg2020multiple} for a selection of papers considering $k$-domination.

A set $D\subseteq V(G)$ is a \textit{$k$-tuple dominating set} of $G$ if the closed neighborhood $N_G[u]$ of each vertex $u \in V(G)$ has at least $k$ vertices in $D$ (or, equivalently, if $|N_G[u]\cap D|\geq k$ for each $u\in V(G)$). The \textit{$k$-tuple domination number} of $G$ is the minimum cardinality of a $k$-tuple dominating set of $G$, and is denoted by $\ktuple{k}(G)$. A {\em $\ktuple{k}$-set} of $G$ is a $k$-tuple dominating set of cardinality $\ktuple{k}(G)$. See~\cite{harary2000double,hansberg2020multiple} for a selection of papers on $k$-tuple domination.

Finally, a set $D\subseteq V(G)$ is a \textit{$k$-tuple total dominating set} of $G$ if the open neighborhood of each vertex $u\in V(G)$ has at least $k$ vertices in $D$ (that is, if $|N_G(u)\cap D|\geq k$ for each $u\in V(G)$). The \textit{$k$-tuple total domination number} of $G$ is the minimum cardinality of a $k$-tuple total dominating set of $G$, and is denoted by $\ktuplet{k}(G)$. As usual, a {\em $\ktuplet{k}$-set} of $G$ stands for a $k$-tuple total dominating set of $G$ of cardinality $\ktuplet{k}(G)$. The $k$-tuple total domination number was studied, for example, in~\cite{henningkazemi, henningyeo-10}, while all three invariants were studied for the case $k=2$ in~\cite{bonomo2018}.

Note that $\ktuple{k}(G)$ is defined only in graphs $G$ with $\delta(G)\ge k-1$, while $\ktuplet{k}(G)$ is defined only if $\delta(G)\ge k$, where $\delta(G)$ denotes the minimum degree of vertices in $G$. One can derive immediately from the definitions that 
\begin{equation}
\label{e:invariants}
\kdom{k}(G)\le \ktuple{k}(G) \le \ktuplet{k}(G) 
\end{equation}
holds for any graph $G$ in which the corresponding invariants are defined. 
Clearly, when $k=1$, both $\ktuple{1}(G)$ and $\kdom{1}(G)$ correspond to the well-known domination number $\gamma(G)$, and $\ktuplet{1}(G)$ corresponds to the total domination number $\gamma_t(G)$. In this paper, we denote by $\kdom{k}(n,r)$, $\ktuple{k}(n,r)$ and $\ktuplet{k}(n,r)$ the $k$-domination number, the $k$-tuple domination number and the $k$-tuple total domination number, respectively, of the Kneser graph $\Kneser{n}{r}$.

Concerning the domination number of Kneser graphs, it was shown to be monotonically decreasing when $r$ is fixed and $n$ grows~\cite{gorodezky2007dominating} and a similar result was proved for the $k$-tuple domination number in~\cite{cornet2023k}. In this vein we prove in Section~\ref{sec:monotone} that $\ktuplet{k}(n,r)\ge \ktuplet{k}(n+1,r)$ and $\kdom{k}(n,r)\ge \kdom{k}(n+1,r)$ hold for any positive integers $n$ and $r$, where $n\ge 2r+1$, respectively $n\ge 2(k+r)$. In Section~\ref{sec:exact}, we obtain exact values for these two invariants in $K(n,r)$ when $n$ is sufficiently large with respect to $r$. Notably, we prove that $\kdom{k}(n,r)=\ktuplet{k}(n,r)=k+r$ as soon as $n\geq r(k+r)$, and, in addition, every $\kdom{k}$-set and $\ktuplet{k}$-set is a clique in this case. When $n$ is one less, we also get exact values, namely, if $k\geq 2$, we have $\kdom{k}(r(k+r)-1,r)=\ktuplet{k}(r(k+r)-1,r)=k+r+1$. In Section~\ref{sec:packing}, we study the 2-packing number, $\rho_2(K(n,r))$, of Kneser graphs, which is the largest cardinality of a set of vertices in the graph, which are pairwise at distance at least $3$ apart. We continue the study from~\cite{cornet2023k}, where the $2$-packing number was used for bounding the $k$-tuple domination number of $K(n,r)$, and $\rho_2(K(3r-2,r))$ was also determined for all values of $r$. Here we go a step further by considering $\rho_2(K(3r-3,r))$, where we obtain lower bounds when $r\le 8$, prove that $4$ is the exact value for $r=9$, and that $3$ is the exact value when $r\ge 10$. We also give sufficient conditions for $\rho_2(K(n,r))$ being equal to $3$, resp.~$4$, by imposing bounds on $r$ with respect to $n$. In addition, we prove a version of monotonicity, by showing that $\packing(K(n+1,r+1))\geq \packing(K(n,r))$ as soon as $n\geq 2r+2$. 

We conclude the introduction with several useful definitions. For $r,n\in\N$ with $r\leq n$, let $[r..n]$ and $[n]$ denote the sets $\{r,\ldots,n\}$ and $\{1,\ldots,n\}$ respectively. 
Given a set of vertices $D$ in $\Kneser{n}{r}$ and $x\in [n]$, the occurrences of the element $x$ in $D$, denoted by $i_x(D)$, represent the number of vertices in $D$ that contain the element $x$. This is, $i_x(D)=|\{u\in D : x\in u\}|$. 
For a positive integer $a$, we define $X_a(D)$ as the set of elements in $[n]$ such that their occurrences in $D$ is equal to $a$, i.e., $X_a(D)=\{x\in [n] : i_x(D)=a\}$. Similarly, we define $X_a^\geq(D)=\{x\in [n] : i_x(D)\geq a\}$, and $X_a^\leq(D)=\{x\in [n] : i_x(D)\leq a\}$. 
When the set $D$ is clear from the context, we shall omit it in the notation. It is important to note that the sum of the occurrences of all elements in $D$ is equal to $r$ times the cardinality of $D$, i.e., $\sum_{x\in [n]} i_x(D)=r|D|$.

\section{Monotonicity}
\label{sec:monotone}

In \cite{cornet2023k} it is proved that the function $\ktuple{k}(n,r)$ is decreasing with respect to $n$. Using the same idea, we prove the analogous statement for the function $\ktuplet{k}(n,r)$. Regarding the $k$-domination number, we show that $\kdom{k}(n,r)$ is decreasing with respect to $n$ from $n_0=2(k+r)$. 

Note that, using the standard notation for vertices of Kneser graphs, a vertex $u\in V(K(n,r))$ also belongs to $K(n+1,r)$, since it is represented as an $r$-subset of $[n]\subset [n+1]$. When $r$ is fixed, we will simplify the notation by writing $N_n(u)$ for the neighborhood $N_{K(n,r)}(u)$, where $u\subset [n]$ with $|u|=r$. In this sense, the meaning of $N_{n+1}(u)$ should also be clear.

\begin{theorem}
    For any positive integers $n$ and $r$, where $n\ge 2r+1$, 
    $$\ktuplet{k}(n,r)\ge \ktuplet{k}(n+1,r).$$ That is, $\ktuplet{k}(n,r)$ is decreasing with respect to $n$. 
\begin{proof}
    Let $D$ be a $\ktuplet{k}$-set of $\Kneser{n}{r}$. Let us show that $D$ is a $k$-tuple total dominating set of $\Kneser{n+1}{r}$. Let $u\in V(\Kneser{n+1}{r}$.
    If $u\in V(\Kneser{n+1}{r})\cap V(\Kneser{n}{r})$, then $|N_{n+1}(u)\cap D|=|N_{n}(u)\cap D|\geq k$.
    If $u\in V(\Kneser{n+1}{r})\setminus V(\Kneser{n}{r})$, then we have $u=\tilde{u}\cup\{n+1\}$ with $\tilde{u}\subseteq [n]$, $|\tilde{u}|=r-1$. Let $x\in [n]\setminus\tilde{u}$, and let $w=\tilde{u}\cup\{x\}$. Since $w\in V(\Kneser{n+1}{r})\cap V(\Kneser{n}{r})$, then $|N_{n+1}(w)\cap D|\geq k$. On the other hand,
    $$N_{n+1}(w)\cap D=\{v\in D:v\cap w=\emptyset\}\subseteq\{v\in D:v\cap \tilde{u}=\emptyset\}=\{v\in D:v\cap u=\emptyset\}.$$
    Thus, we have $|N_{n+1}(u)\cap D|\geq |N_{n+1}(w)\cap D|\geq k$.
    Therefore, $D$ is a $k$-tuple total dominating set of $\Kneser{n+1}{r}$. Consequently, $\ktuplet{k}(n+1,r)\leq |D|=\ktuplet{k}(n,r)$.
\end{proof}
\end{theorem}

\begin{lemma}\label{lem: monot kdom}
    If $n\geq 2(k+r)$, $D$ is a $\kdom{k}$-set of $\Kneser{n}{r}$ and $\tilde{u}=\{a_1,\ldots,a_{r-1}\}\subseteq[n]$, then there exists $x\in[n]\setminus\tilde{u}$ such that $\tilde{u}\cup\{x\}\notin D$. 
\begin{proof}
    Suppose that $\tilde{u}\cup\{x\}\in D$ for each $x\in[n]\setminus\tilde{u}=\{x_1,x_2,\ldots,x_{n-r+1}\}$, and let $u_i=\tilde{u}\cup \{x_i\}$ for each $i\in[n-r+1]$. Let $D'=(D\setminus\{u_1,\ldots,u_{k+1}\})\cup\{v_1,\ldots,v_k\}$, where $v_i=\{x_{k+2},\ldots,x_{k+r},x_{k+r+i}\}$ for $i\in[k]$. It is possible since $n\geq 2(k+r)$.

    Let us see that $D'$ is a $k$-dominating set of $\Kneser{n}{r}$. Let $w\in V(\Kneser{n}{r})\setminus D'$.
    \begin{itemize}
        \item If $w\cap \tilde{u}=\emptyset$ (and $w\neq u_i$ for each $i\in[k+1]$), then as $|w\cap\{x_{k+2},\ldots,x_{n-r+1}\}|\leq r$, we have $|N(w)\cap D'|\geq |\{u_i:i\geq k+2\wedge x_i\notin w\}|\geq |\{u_i:i\geq k+2\}|-r=(n-r-k)-r\geq k$ since $n\geq 2(k+r)$.
        
        \item If $w\cap \tilde{u}\neq\emptyset$ and $w\neq u_i$ for each $i\in[k+1]$, then $|N(w)\cap D'|\geq |N(w)\cap D|\geq k$, since the vertices eliminated from $D$ were not neighbors of $w$.

        \item If $w=u_i$ for some $i\in[k+1]$, then we have $|N(w)\cap D'|\geq|\{v_1,\ldots,v_k\}|=k$.
    \end{itemize}
    Therefore, $D'$ is a $k$-dominating set of $\Kneser{n}{r}$ of size $|D'|=|D|-1=\kdom{k}(n,r)-1$, and we arise to a contradiction. Thus, there is at least one element $x\in[n]\setminus\tilde{u}$ such that $\tilde{u}\cup \{x\}\notin D$.
\end{proof}
\end{lemma}

\begin{theorem}\label{thm: monot kdom}
    For any positive integers $n$ and $r$, where $n \geq 2(k+r)$, $$\kdom{k}(n,r)\geq \kdom{k}(n+1,r).$$
\begin{proof}
    Let $n\geq 2(k+r)$ and let $D$ be a $\kdom{k}$-set of $\Kneser{n}{r}$. Let us show that $D$ is a $k$-dominating set of $\Kneser{n+1}{r}$. Let $u\in V(\Kneser{n+1}{r}) \setminus D$. If $u\subseteq [n]$, then $|N_{n+1}(u)\cap D|=|N_n(u)\cap D|\geq k$. Otherwise, $u\in V(\Kneser{n+1}{r})\setminus V(\Kneser{n}{r})$. It follows that $u=\tilde{u}\cup\{n+1\}$ with $\tilde{u}=\{a_1,\ldots,a_{r-1}\}\subseteq [n]$. By Lemma \ref{lem: monot kdom} there exists $x\in[n]\setminus \tilde{u}$ such that $w=\tilde{u}\cup\{x\}\notin D$. Thus, we have
    \begin{align*}
        N_{n}(w)\cap D&=N_{n+1}(w)\cap D=\{v\in D:v\cap w=\emptyset\}=\\
        &=\{v\in D:v\cap \tilde{u}=\emptyset \wedge x\notin v\}\subseteq\{v\in D:v\cap\tilde{u}=\emptyset\}=N_{n+1}(u)\cap D.
    \end{align*}
    Consequently,
    $$|N_{n+1}(u)\cap D|\geq |N_{n}(w)\cap D|\geq k.$$
    Therefore, $D$ is a $k$-dominating set of $\Kneser{n+1}{r}$ and we have
    $$\kdom{k}(n+1,r)\leq |D|=\kdom{k}(n,r).$$
\end{proof}
\end{theorem}

Note that $\kdom{k}(n,r)$ is not necessarily decreasing for any $n$, but $\kdom{k}(n_1,r) \geq \kdom{n}(n_2,r)$ holds for any $n_1 \leq n_2$, if $n_1$ is large enough. To see that the monotonicity may not hold if $n_1$ is small, consider the Kneser graphs $K(n,2)$, where $\kdom{2}(5,2) < \kdom{2}(6,2)$, $\kdom{2}(6,2)=\kdom{2}(7,2)$ and $\kdom{2}(7,2)>\kdom{2}(8,2)$, see Table~\ref{tab: K(n,2)}.

\section{\texorpdfstring{Exact values for large $n$}{Large n}}
\label{sec:exact}

In \cite{cornet2023k}, the following result is stated.

\begin{theorem}[\cite{cornet2023k}]\label{thm: ktuple large n}
For any $n\ge 2r$,    $\ktuple{k}(n,r)=k+r$ if and only if $n\geq r(k+r)$.
\end{theorem}

Moreover, it is shown that except in the case $k=1$ and $r=2$, every $\ktuple{k}$-set of $K(n,r)$, where $n \geq r(k+r)$, is a clique. If $k=1$ and $r=2$, then in \cite{ivanvco1993domination} it is proved that a $\ktuple{1}$-set of $K(n,2)$ is either a clique or an independent set. In any case, if $\ktuple{k}(n,r)=k+r$, then it is possible to obtain a $\gamma_k$-set that is a clique. As a by-product, we get the following for the $k$-tuple total domination number.

\begin{theorem}
    For any $n\ge 2r$, $\ktuplet{k}(n,r)=k+r$ if and only if $n\geq r(k+r)$. Moreover, for any $n \geq r(k+r)$, every $\ktuplet{k}$-set is a clique.
\begin{proof}
    Let $n,r,k\in\N$. If $n\geq r(k+r)$, then $\ktuplet{k}(n,r)\geq \ktuple{k}(n,r)=k+r$. Since there exists a $\ktuple{k}$-set $D$ which is a clique on $k+r$ vertices, it is also a $k$-tuple total dominating set of $\Kneser{n}{r}$. Thus, $\ktuplet{k}(n,r)=k+r$.  Conversely, if $\ktuplet{k}(n,r)=k+r$, then since any $\ktuplet{k}$-set $D$ is a $k$-tuple dominating set, then $\ktuple{k}(n,r)\leq |D|=k+r$. Due to monotonicity of $\ktuple{k}(n,r)$ with respect to $n$, we have $\ktuple{k}(n,r)=k+r$, and by Theorem \ref{thm: ktuple large n} it follows that $n\geq r(k+r)$. Hence $D$ is a $\ktuple{k}$-set of $K(n,r)$ for $n \geq r(k+r)$ and thus $D$ is a clique by the results described before the theorem.    
\end{proof}
\end{theorem}

\begin{remark}\label{rem: condition kdom}
    If $D$ is a $k$-dominating set of the Kneser graph $\Kneser{n}{r}$ and $w\in V(\Kneser{n}{r})\setminus D$, then $w$ has nonempty intersection with at most $|D|-k$ vertices of $D$. Otherwise, $|N(w)\cap D|<k$ contradicting the fact that $D$ is $k$-dominating.
\end{remark}

\begin{theorem}\label{thm: kdom for large n}
    If $k\geq 2$ and $n\geq k+2r$, then 
    \begin{itemize}
        \item $\kdom{k}(n,r)=k+r$ if and only if $n\geq r(k+r)$; moreover, every $\kdom{k}$-set is a clique;
        \item $\kdom{k}(n,r)\geq k+r+1$ if and only if $n<r(k+r)$.
    \end{itemize}
\begin{proof}
    Let $n,r,k\in\N$ such that $n\geq k+2r$. We start the proof by showing that
    \begin{equation}\label{eq:2}
        \kdom{k}(n,r) \leq k+r 	\Rightarrow n \geq r \cdot (k+r).
    \end{equation}
     
    Thus, suppose that $\kdom{k}(n,r) \leq k+r$  and let $D$ be an arbitrary $k$-dominating set of $\Kneser{n}{r}$ of cardinality $|D|=k+r$ (note that such a set exists, since $n\geq k+2r$, and thus we can obtain $D$ by adding $k+r-\kdom{k}(n,r)$ vertices to a $\kdom{k}(n,k)$-set).
    We will show that $D$ is a clique or, equivalently, that the vertices in $D$ are pairwise disjoint. Assume $D$ contains two adjacent vertices $u$ and $v$. Hence there exists $j \in [n]$ such that $j \in u \cap v$ and consequently $i_j(D) \geq 2$. Let $a\in [n]$ such that $i_a=\max\{i_x:x\in [n]\}$. Thus, $i_a\geq i_j \geq 2$.
    Let us see that $i_a=2$. Otherwise, let $u_1,u_2,u_3\in D$ such that $a\in u_1\cap u_2\cap u_3$, and $u_4,\ldots,u_{r+1}\in D\setminus\{u_1,u_2,u_3\}$. Consider
    \begin{align*}
        b_1&=a\in u_1\cap u_2\cap u_3,\\
        b_i&\in u_{i+2}\setminus\{b_1,\ldots,b_{i-1}\},\ \text{for }i\in[2..r-1],
    \end{align*}
    and let $b=\{b_1,\ldots,b_{r-1}\}$. Observe that there is at least one element $x\in[n]\setminus b$ such that $b\cup\{x\}\notin D$, since $|[n]\setminus b|-|D|=(n-r+1)-(k+r)=n-(k+2r)+1\geq 1$. Let $w=b\cup\{x\}$. We have $w\notin D$ and $w\cap u_i\neq\emptyset$ for every $i\in[r+1]$ which cannot be true because of Remark \ref{rem: condition kdom}. Thus, $i_a=2$. Let $u_1$ and $u_2$ be the two vertices of $D$ that contain the element $a$.
    
    If $r>2$, let $u_3,\ldots,u_{r+1}\in D\setminus\{u_1,u_2\}$. Consider
    \begin{align*}
        b_1&=a\in u_1\cap u_2,\\
        b_i&\in u_{i+1}\setminus\{b_1,\ldots,b_{i-1}\},\ \text{for }i\in[2..r],
    \end{align*}
    and let $w=\{b_1,\ldots,b_{r}\}$. We have $w\cap u_i\neq\emptyset$ for every $i\in[r+1]$. By Remark \ref{rem: condition kdom}, it follows that $w\in D$ and as $a\in w$, then $w$ is either $u_1$ or $u_2$. Without loss of generality, $w=u_1$. Note that $b_i\in u_1\cap u_{i+1}$ for every $i$, and since $i_a=\max\{i_x:x\in[n]\}=2$, $b_i\notin u_j$ for every $j$ different from $1$ and $i+1$.
    Let us consider $x\in u_3\setminus u_1$, and let $w'=w\setminus\{b_2\}\cup\{x\}$. We have $x\in w'\setminus u_1$ and $b_r\in w'\setminus u_2$. Thus, $a\in w'$ and $w'\notin \{u_1,u_2\}$ and therefore $w'\notin D$ but $w'\cap u_i\neq \emptyset$ for every $i\in[r+1]$, arising to a contradiction with Remark~\ref{rem: condition kdom}.

    If $r=2$, we have $u_1=\{a,b\}$ and $u_2=\{a,c\}$ for some $b,c\in [n]$. Since $|D|=k+r\geq 4$, then we can choose $u_3\in D$ $u_3\neq \{b,c\}$. Let $x\in u_3\setminus\{b,c\}$. Since $i_a=2$, $w=\{a,x\}\notin D$ and $w\cap u_i\neq\emptyset$ for $i\in[3]$, contradicting the fact that $D$ is $k$-dominating.

    Therefore, if $\kdom{k}(n,r) \leq k+r$ and $D$ is a $k$-dominating set of $\Kneser{n}{r}$ of cardinality $k+r$, then the vertices in $D$ are pairwise disjoint and as a consequence, $n\geq r|D|=r(k+r)$. In particular, if $\kdom{k}(n,r)=k+r$, then $n\geq r(k+r)$.

    \medskip

    Conversely, suppose that $n\geq r(k+r)$. We have $\kdom{k}(n,r)\leq\ktuple{k}(n,r)=k+r$. Suppose that $\kdom{k}(n,r)<k+r$ and let $D$ be a $k$-dominating set of cardinality $|D|=k+r-1$. Let $u_1,\ldots,u_r\in D$. Consider
    \begin{align*}
        b_1&\in u_1,\\
        b_i&\in u_{i}\setminus\{b_1,\ldots,b_{i-1}\},\ \text{for }i\in[2..r],
    \end{align*}
    and let $w=\{b_1,\ldots,b_{r}\}$. Let us note that $w\cap u_i\neq \emptyset$ for every $i\in[r]$. So, by Remark \ref{rem: condition kdom} it follows that $w\in D$. Let us note that we may assume that $w=u_j$ for some $j\in [r]$ by changing our initial choice of the vertices $u_i$. Notice that there exists $x\in [n]\setminus\left(\bigcup_{u\in D} u\right)$, since $\left|[n]\setminus\left(\bigcup_{u\in D} u\right)\right|\geq n-r|D|=n-r(k+r-1)\geq r$. Let $w'=w\setminus\{b_j\}\cup\{x\}$. We have $w'\notin D$ due to our choice of $x$, and $w'\cap u_i\neq\emptyset$ for each $i\in [r]$ (note that $u_j \cap w'=\{b_1,\ldots , b_r\}\setminus \{b_j\}$), arising to a contradiction with Remark~\ref{rem: condition kdom}.

    Therefore, if $n\geq r(k+r)$, then $\kdom{k}(n,r)=k+r$.

    Suppose now that $\kdom{k}(n,r) \geq k+r+1$. Then the first statement of the theorem implies that $n<r(k+r)$. For the converse assume that $n<r(k+r)$. It follows from the first statement of this theorem, that $\kdom{k}(n,k) \neq k+r$. Suppose, to the contrary, that $\kdom{k}(n,r) < k+r$. Then it follows from Equation~\ref{eq:2} that $n \geq r(r+k)$, a contradiction.
\end{proof}
\end{theorem}

\begin{corollary}
    If $k\geq 2$ and $n\geq r(k+r)$, then
    $$\kdom{k}(n,r)=\ktuple{k}(n,r)=\ktuplet{k}(n,r)=k+r.$$
    Moreover, every $\kdom{k}$-set, $\ktuple{k}$-set and $\ktuplet{k}$-set of $K(n,r)$, where $n \geq r(k+r)$, is a clique.
\end{corollary}

\begin{proposition} If $k\geq 2$, then
    $$\kdom{k}(r(k+r)-1,r)=\ktuple{k}(r(k+r)-1,r)=\ktuplet{k}(r(k+r)-1,r)=k+r+1.$$
\begin{proof}
    Let $k\geq 2$ and $n=r(k+r)-1$. Since $n>k+2r$, from Theorem \ref{thm: kdom for large n} we have that $k+r+1\leq \kdom{k}(n,r)\leq \ktuple{k}(n,r)\leq \ktuplet{k}(n,r)$. We will give a $k$-tuple total dominating set $D$ of cardinality exactly $k+r+1$.
    
    It is enough to consider the set $D=A\cup B$, where
    \begin{align*}
        A&=\big\{[r],[r-1]\cup\{r+1\},[r..2r-1]\big\}\\
        B&=\big\{[2r..3r-1],[3r..4r-1],\ldots,[(k+r-1)r..(k+r)r-1]\big\}.
    \end{align*}
    We have that $A$ is an independent set of $|A|=3$ vertices, $B$ is a clique of $|B|=k+r-2$ vertices, and $a$ is adjacent to $b$ for each $a\in A$ and $b\in B$. Let $u\in V(\Kneser{n}{r})$.
    \begin{itemize}
        \item If $u\in A$, then $u$ is adjacent to every vertex in $B$, so $|N(u)\cap D|\geq k+r-2\geq k$.
        \item If $u\in B$, then $u$ is adjacent to every other vertex in $D$, so $|N(u)\cap D|\geq k+r> k$.
        \item If $u\notin D$ and $|u\cap [2r-1]|\geq 2$, then $|u\cap [2r..n]|\leq r-2$ and $u$ has at least $|B|-(r-2)=k$ neighbors in $B$. So, $|N(u)\cap D|\geq k$.
        \item If $u\notin D$ and $|u\cap [2r-1]|=1$, then $|u\cap [2r..n]|=r-1$ and $u$ has at least $|B|-(r-1)=k-1$ neighbors in $B$. On the other hand, since there is no element contained in every vertex of $A$, we have that $u$ has at least one neighbor in $A$. So, $|N(u)\cap D|=|N(u)\cap A|+|N(u)\cap B|\geq k$.
        \item Finally, if $u\notin D$ and $|u\cap [2r-1]|=0$, then $|u\cap [2r..n]|=r$ and $u$ has at least $|B|-r=k-2$ neighbors in $B$, and $u$ is adjacent to every vertex in $A$. So, $|N(u)\cap D|=|N(u)\cap A|+|N(u)\cap B|\geq k+1\geq k$.
    \end{itemize}

\end{proof}
\end{proposition}

Combining the results of this section with results from~\cite{cornet2023k}, we obtain the table of values of the three types of $k$-domination invariants of $K(n,r)$ in the case $k=2$ and $r=2$; see Table~\ref{tab: K(n,2)}.

\begin{table}[H]
    \centering
    \begin{tabular}{|c|c|c|c|c|} \hline
        $n$ & $\kdom{2}(n,2)$ & $\ktuple{2}(n,2)$ & $\ktuplet{2}(n,2)$ \\ \hline
        $4$ & $6$ & $6$ & $\lnot\exists$ \\
        $5$ & $4$ & $6$ & $8$ \\
        $6$ & $5$ & $6$ & $6$ \\
        $7$ & $5$ & $5$ & $5$ \\
        $\geq 8$ & $4$ & $4$ & $4$ \\ \hline
    \end{tabular}
    \caption{Domination invariants with $k=2$ for $\Kneser{n}{2}$.}
    \label{tab: K(n,2)}
\end{table}

\section{2-packing number}
\label{sec:packing}

The $2$-packing number of Kneser graphs was considered in~\cite{cornet2023k}. Note that ${\rm diam}(K(n,r))=2$ as soon as $n\ge 3r-1$, and in such cases $\rho_2(K(n,r))=1$; see~\cite{valencia2005diameter}. In~\cite{cornet2023k}, the authors studied the case, which is in a sense the closest to diameter $2$ Kneser graphs, and this is when $n=3r-2$. They obtained the exact values of the $2$-packing number for all these Kneser graphs. In this paper, we consider the next case, which is when $n=3r-3$. We again simplify the notation by writing $\rho_2(n,r)$ instead of $\rho_2(K(n,r))$.

We start by recalling a useful observation from~\cite{cornet2023k}, which follows from the fact that a set $S$ is a $2$-packing of a graph $G$ if and only if no two vertices of $S$ are adjacent nor have a common neighbor. 

\begin{observation}
\label{obs:pack}
Let $2r+1\le n\le 3r-2$. A set $S$ is a $2$-packing in $K(n,r)$ if and only if for every pair $u,v\in S$ we have $1\le |u\cap v|\le (3r-1)-n$.
\end{observation}

In the case of $K(3r-3,r)$, Observation~\ref{obs:pack} yields that $S$ is a $2$-packing in $K(n,r)$ if and only if for every pair $u,v\in S$ we have $$1\le |u\cap v|\le 2.$$
We obtain lower bounds for $\rho_2(3r-3,r)$ when $r\le 9$, and exact values for $\rho_2(3r-3,r)$ when $r>9$. See Table~\ref{tab:2pack}. 
The lower bounds in Table~\ref{tab:2pack} follow from constructions of the corresponding $2$-packings, which we present in Table~\ref{tab:2pack-values}.

\begin{table}[H]
    \centering
    \begin{tabular}{|c|c|} \hline
        $r$ & $\rho_2(3r-3,r)$ \\ \hline
        $4$ & $\ge 12$ \\
        $5$ & $\ge 12$\\
        $6$ & $\ge 10$  \\
        $7$ & $\ge 6$ \\
        $8$ & $\ge 5$ \\
        $9$ & $=4$ \\
        $\ge 10$ & $=3$ \\         \hline
    \end{tabular}
    \caption{$2$-packing numbers of $\Kneser{3r-3}{r}$.}
    \label{tab:2pack}
\end{table}

Consider the Kneser graph $K(3r-3,r)$, where $r\ge 2$. Since a set $S$ of three vertices $u,v,w \in V(K(3r-3,r))$ with $|u \cap v \cap w|=|u\cap v|=|u \cap w|=|v \cap w|=2$ is a 2-packing (say $u=[1..r], v=\{1,2\} \cup [r+1..2r-2], w=\{1,2\} \cup [2r-1..3r-4]$), we infer $\rho_2(3r-3,r)\ge 3$. The next result proves the last line of Table~\ref{tab:2pack}.
\begin{proposition}
\label{prp:3r-3}
If $r\ge 10$, then $\rho_2(3r-3,r)=3$.    
\end{proposition}
\begin{proof}
   Let $P$ be a $2$-packing of $K(3r-3,r)$, where $r\ge 10$, and let $\{u_1,u_2,u_3\}\subseteq P$. Since $|u_i\cap u_j|\le 2$ for all $\{i,j\}\subset [3]$, we infer that $|[3r-3]\setminus (u_1\cup u_2\cup u_3)|\le 3$. Suppose that $|P|\ge 4$, and let $z\in P$ be distinct from all $u_i$. Since $z\in P$, $|z\cap u_1|\le 2$, $|z\cap u_2|\le 2$, $|z\cap u_3|\le 2$, and $|z\cap \left([3r-3]\setminus (u_1\cup u_2\cup u_3)\right)|\le 3$. We derive that $|z|\le 2+2+2+3=9$, which is a contradiction, since $r\ge 10$. Hence, $|P|=3$.
\end{proof}

\begin{table}[H]
    \centering
    \begin{tabularx}{\linewidth}{|c|C|} \hline
        $r$ & $2$-packing of $K(3r-3,r)$ \\ \hline
        $4$ & (1, 2, 3, 5), (1, 2, 6, 9), (1, 2, 7, 8),  (1, 3, 4, 6), (1, 4, 5, 8), (1, 4, 7, 9),(2, 3, 4, 7), \\ &(2, 4, 5, 6), (2, 4, 8, 9), (3, 5, 7, 9), (3, 6, 8, 9), (5, 6, 7, 8)\\
            \hline
        $5$ &(1, 2, 3, 4, 8), (1, 2, 5, 10, 11), (1, 2, 6, 9, 12), (1, 3, 7, 9, 10), (1, 4, 5, 7, 12), \\ &(1, 6, 7, 8, 11), (2, 3, 5, 6, 7), (2, 4, 7, 9, 11), (2, 7, 8, 10, 12), (3, 4, 6, 10, 11), \\ &(3, 5, 9, 11, 12), {(4, 5, 6, 8, 9)} \\ \hline
        $6$ &  (1, 2, 3, 4, 5, 11), (1, 2, 7, 8, 9, 14), (1, 3, 6, 9, 10, 15), (1, 5, 6, 12, 13, 14), \\ &(2, 4, 6, 7, 13, 15), (2, 5, 8, 10, 12, 15), (3, 4, 8, 10, 13, 14), (3, 6, 7, 8, 11, 12), \\ &(4, 9, 11, 12, 14, 15), (5, 7, 9, 10, 11, 13)\\ \hline
        $7$ &  (1, 2, 3, 4, 6, 11, 14), (1, 3, 5, 8, 16, 17, 18), (2, 4, 5, 7, 12, 15, 17), \\ &(2, 8, 9, 12, 13, 14, 16), (3, 7, 8, 9, 10, 11, 15), (4, 5, 6, 9, 10, 13, 18),  \\ \hline
        $8$ & (1, 2, 3, 4, 5, 6, 9, 18), (1, 2, 7, 11, 12, 14, 20, 21), (3, 7, 8, 9, 13, 15, 16, 20), \\ &(4, 5, 10, 12, 13, 15, 17, 21), (5, 6, 8, 10, 11, 14, 16, 19) \\ \hline
        %$9$ & (1, 2, 3, 4, 5, 8, 10, 11, 12), (1, 2, 6, 7, 16, 17, 18, 19, 21),  \\ &  (3, 4, 6, 7, 9, 13, 14, 15, 20), (5, 8, 9, 13, 16, 17, 22, 23, 24) \\ \hline
    \end{tabularx}
    \caption{$2$-packings in $\Kneser{3r-3}{r}$, for $r\leq 8$.}
    \label{tab:2pack-values}
\end{table}

We suspect that the lower bounds in Table~\ref{tab:2pack} are in fact exact values of $\rho_2(3r-3,r)$, and leave this as an open problem. The proof for the penultimate line in Table~\ref{tab:2pack}, as well as an alternative proof of Proposition \ref{prp:3r-3}, are given afterwards in Theorem \ref{thm: packing 3 and 4}.

%We will use Observation~\ref{obs:pack} in all further results of this section.
We continue with several general results concerning the behavior of the $2$-packing number in Kneser graphs.

\begin{proposition}\label{prp: 2packingIncreasing}
    If $n$, $r$ and $a$ are positive integers such that $n=2r+1$ and $a\geq 2$, then
    $$\packing(n+2a,r+a)\geq 2\packing(n,r).$$
\end{proposition}
    
    \begin{proof}
        Let $S$ be a $2$-packing of $K(n,r)$ with $|S|=\rho_2(n,r)$. Consider the following set: 
        $$T=\Big\{v\cup[(n+1)..(n+a)],v\cup[(n+a+1)..(n+2a)]:\,v\in S\Big\}.$$ 
        Clearly, $T\subset V(K(n+2a,r+a))$ and $|T|=2\rho_2(n,r)$. Note that $n+2a=2r+1+2a=2(r+a)+1$. Thus, to prove that $T$ is a $2$-packing in $K(n+2a,r+a)$, we need to prove, by Observation~\ref{obs:pack}, that 
        $$1\le |u\cap v|\le (3(r+a)-1)-(n+2a)=3r-1-n+a=r+a-2$$
        for any two vertices $u,v\in T$. Let $u,v\in T$, and let $\tilde{u}=u\cap [n]$, $\tilde{v}=v\cap [n]$. If $\tilde{u}\neq\tilde{v}$, then as $\tilde{u},\tilde{v}\in S$, by Observation~\ref{obs:pack}, we have $1\le |\tilde{u}\cap \tilde{v}|\le (3r-1)-n=r-2$, and since $|\tilde{u}\cap \tilde{v}|\leq |u\cap v|\leq |\tilde{u}\cap \tilde{v}|+a$, we have $1\le|u\cap v|\le r+a-2$. On the other hand, if $\tilde{u}=\tilde{v}$, then $|u\cap v|=r$ and it follows $1\leq|u\cap v|=r\le r+a-2$ since $a\ge 2$.
    \end{proof}

By applying Proposition~\ref{prp: 2packingIncreasing} several times (note that if $K(n,r)$ is an odd graph, then $K(n+2a,r+a)$ is also an odd graph) we get the following.

\begin{corollary}
    If $n$, $r$ and $a$ are positive integers such that $n=2r+1$ and $a\geq 2$, then
    $$\packing(n+2a,r+a)\geq 2^{\left\lfloor\frac{a}{2}\right\rfloor}\packing(n,r).$$
\end{corollary}

We follow with a version of monotonicity of the $2$-packing number in Kneser graphs, which turns out to be non-decreasing when both $n$ and $r$ increase by the same value. 

\begin{theorem}
     If $n$ and $r$ are positive integers such that $n\geq 2r+2$, then $$\packing(n+1,r+1)\geq \packing(n,r).$$
\end{theorem}
   
    \begin{proof}
        Note that if $n\geq 3r-1$, then $\packing(n,r)=1$ and the result is straightforward. Thus, let us assume $2r+2\leq n\leq 3r-2$. Let $S$ be a $2$-packing of $\Kneser{n}{r}$ of size $\packing(n,r)$. Consider the set:
        $$T=\Big\{v\cup\{n+1\}:\,v\in S\Big\}.$$ 
        We have $T\subset V(K(n+1,r+1))$ and $|T|=\packing(n,r)$. Let us prove that $T$ is a 2-packing of $\Kneser{n+1}{r+1}$. In order to do so, since $2(r+1)+1\leq n+1\leq 3(r+1)-2$, by Observation \ref{obs:pack} it is enough to see that $1\leq |u\cap v|\leq (3(r+1)-1)-(n+1)=3r+1-n$ for any pair of vertices in $T$. In fact, let $u,v\in T$, and let $\tilde{u}=u\cap[n]$, $\tilde{v}=v\cap[n]$. Since $\tilde{u},\tilde{v}\in S$, we have $1\leq |\tilde{u}\cap\tilde{v}|\leq 3r-1-n$. As $|u\cap v|=|\tilde{u}\cap \tilde{v}|+1$, the result follows.
    \end{proof}

%\bb{Concerning the lemma below: note that we have two conditions for the upper bound on $t$ with respect to $r$; firstly, $t\leq r-1$, and secondly, $t\le \frac{r+3}{3}$. The former is smaller only in the case $r=2$, which is not really interesting. I think we may assume that $r\ge 3$, and that $2\le t\le \frac{r+3}{3}$. This does not have an effect on the proof, just perhaps in the second sentence, where the bound $t\leq r-1$ is mentioned.}

\begin{lemma}\label{lem: packing of size 4}
    %Let $r$, $n$ and $t$ be positive integers such that $2\leq t\leq r-1$, $n=3r-t$ and $r\geq 3(t-1)$. If $\packing(n,r)\geq 4$, then there exists a $2$-packing $S$ of $\Kneser{n}{r}$ with $|S|=4$ and $i_x(S)\leq 2$ for every $x\in[n]$.
    Let $r$, and $t$ be positive integers such that $r\geq 3$ and $2\leq t\leq \frac{r+3}{3}$. If $\packing(3r-t,r)\geq 4$, then there exists a $2$-packing $S$ of $\Kneser{3r-t}{r}$ with $|S|=4$ and $i_x(S)\leq 2$ for every $x\in[n]$.
    \begin{proof}
        Let $n=3r-t$ and let $S$ be a $2$-packing of $\Kneser{n}{r}$ with cardinality $|S|=4$. Note that since $r\geq 3$, we have $\frac{r+3}{3}\leq r-1$. Thus $2\leq t\leq r-1$ and it follows that $2r+1\leq n\leq 3r-2$. So, by Observation \ref{obs:pack}, we have that every pair of vertices in $S$ intersect in at most $3r-1-n=t-1$ elements. Let us assume $X_3^\geq(S) \neq\emptyset$ and let $x\in X_3^\geq(S)$. 

        If $x\in X_3(S)$, let $u_1,u_2,u_3$ be the vertices in $S$ that contain the element $x$ and $u_4$ the vertex in $S$ which does not contain $x$. Note that
        $$\left|u_1\cap X_2^\geq(S)\right|=\left|\bigcup_{i=2}^4 (u_1\cap u_i)\right|\leq 1+\underbrace{\left|(u_1\cap u_2)\setminus\{x\}\right|}_{\leq t-2}+\underbrace{\left|(u_1\cap u_3)\setminus\{x\}\right|}_{\leq t-2}+\underbrace{\left|u_1\cap u_4\right|}_{\leq t-1}\leq 3(t-1)-1.$$
        Thus, since $t\leq\frac{r+3}{3}$, or equivalently $r\geq 3(t-1)$, there is at least one element $x_1\in u_1\cap X_1(S)$. Analogously, there exist $x_2\in u_2\cap X_1(S)$ and $x_3\in u_3\cap X_1(S)$. Let us consider $S'=\{v_1,v_2,v_3,v_4\}$, where
        $$v_1=(u_1\setminus\{x\})\cup \{x_2\},\qquad v_2=(u_2\setminus\{x\})\cup \{x_3\},\qquad v_3=(u_3\setminus\{x\})\cup \{x_1\},\qquad v_4=u_4.$$
        It is easy to see that for every $i\neq j$ we have $|u_i\cap u_j|=|v_i\cap v_j|$. Thus, $S'$ is a $2$-packing of $\Kneser{n}{r}$ of size $4$ and $|X_3^\geq(S')|=|X_3^\geq(S)|-1$.

        If $x\in X_4(S)$, let $S=\{u_1,u_2,u_3,u_4\}$. Now we have
        $$\left|u_1\cap X_2^\geq(S)\right|=\left|\bigcup_{i=2}^4 (u_1\cap u_i)\right|\leq 1+\sum_{i=2}^4\underbrace{\left|(u_1\cap u_i)\setminus\{x\}\right|}_{\leq t-2}\leq 3(t-1)-2.$$
        Thus, since $r\geq 3(t-1)$, there is at least two elements $x_1^1,x_1^2\in u_1\cap X_1(S)$. Analogously, there exist $x_i^1,x_i^2\in u_i\cap X_1(S)$ for $i=2,3,4$. Let us consider $S'=\{v_1,v_2,v_3,v_4\}$, where
        \begin{align*}
            &v_1=(u_1\setminus\{x\})\cup \{x_4^1\},\\
            &v_2=(u_2\setminus\{x\})\cup \{x_1^1\},\\
            &v_3=(u_3\setminus\{x,x_3^2\})\cup \{x_1^2,x_2^1\},\\
            &v_4=(u_4\setminus\{x,x_4^2\})\cup \{x_2^2,x_3^1\}.
        \end{align*}
        It is easy to see that for every $i\neq j$ we have $|u_i\cap u_j|=|v_i\cap v_j|$. Thus, $S'$ is a $2$-packing of $\Kneser{n}{r}$ of size $4$ and $|X_3^\geq(S')|=|X_3^\geq(S)|-1$.

        In any case, if $X_3^\geq (S)\neq \emptyset$, we build a $2$-packing $S'$ of size $4$ with $|X_3^\geq (S')|<|X_3^\geq (S)|$. If $X_3^\geq (S')\neq \emptyset$, we repeat the procedure until we get a $2$-packing of four vertices for which each element occurs at most twice.
    \end{proof}
\end{lemma}

We can prove, analogously, the following lemma.

%\bb{A similar comment as for the previous lemma. I suggest that give only one condition on $t$, notably, $2 \le t\le \frac{r+4}{4}$. }

\begin{lemma}\label{lem: packing of size 5}
    %Let $r$, $n$ and $t$ be positive integers such that $2\leq t\leq r-1$, $n=3r-t$ and $r\geq 4(t-1)$. If $\packing(n,r)\geq 5$, then there exists a $2$-packing $S$ of $\Kneser{n}{r}$ with $|S|=5$ and $i_x(S)\leq 2$ for every $x\in[n]$.
    Let $r$, and $t$ be positive integers such that $r\geq 3$ and $2\leq t\leq \frac{r+4}{4}$. If $\packing(3r-t,r)\geq 5$, then there exists a $2$-packing $S$ of $\Kneser{3r-t}{r}$ with $|S|=5$ and $i_x(S)\leq 2$ for every $x\in[n]$.
    \begin{proof}
        Let $n=3r-t$ and let $S$ be a $2$-packing of $\Kneser{n}{r}$ with cardinality $|S|=5$. Note that since $r\geq 3$, we have $\frac{r+4}{4}<r-1$. Thus $2\leq t< r-1$ and it follows that $2r+1<n\leq 3r-2$. Then, by Observation \ref{obs:pack}, every pair of vertices in $S$ intersect in at most $3r-1-n=t-1$ elements. Let us assume $X_3^\geq(S) \neq\emptyset$ and let $x\in X_3^\geq(S)$. 

        If $x\in X_3(S)\cup X_4(S)$, let $u_1,u_2,u_3$ be vertices in $S$ that contain the element $x$ and $u_4,u_5$ the remaining vertices in $S$ such that $x\notin u_5$. Note that 
        $$\left|u_1\cap X_2^\geq(S)\right|=\left|\bigcup_{i=2}^5 (u_1\cap u_i)\right|\leq 1+\sum_{i=2,3}\underbrace{\left|(u_1\cap u_i)\setminus\{x\}\right|}_{\leq t-2}+\sum_{i=4,5}\underbrace{\left|u_1\cap u_i\right|}_{\leq t-1}\leq 4(t-1)-1.$$
        Thus, since $t\leq\frac{r+4}{4}$, or equivalently $r\geq 4(t-1)$, there is at least one element $x_1\in u_1\cap X_1(S)$. Analogously, there exist $x_2\in u_2\cap X_1(S)$ and $x_3\in u_3\cap X_1(S)$. Let us consider $S'=\{v_1,v_2,v_3,v_4,v_5\}$, where
        $$v_1=(u_1\setminus\{x\})\cup \{x_2\},\quad v_2=(u_2\setminus\{x\})\cup \{x_3\},\quad v_3=(u_3\setminus\{x\})\cup \{x_1\},\quad v_4=u_4,\quad v_5=u_5.$$
        Note that if $x\in X_4(S)$, then $x\in X_1(S')$. It is easy to see that for every $i\neq j$ we have $|u_i\cap u_j|=|v_i\cap v_j|$. Therefore, $S'$ is a $2$-packing of $\Kneser{n}{r}$ of size $5$ and $|X_3^\geq(S')|=|X_3^\geq(S)|-1$.

        If $x\in X_5(S)$, let $S=\{u_1,u_2,u_3,u_4,u_5\}$. We have
        $$\left|u_1\cap X_2^\geq(S)\right|=\left|\bigcup_{i=2}^5 (u_1\cap u_i)\right|\leq 1+\sum_{i=2}^5\underbrace{\left|(u_1\cap u_i)\setminus\{x\}\right|}_{\leq t-2}\leq 4(t-1)-3.$$
        Thus, since $r\geq 4(t-1)$, there are at least three elements $x_1^1,x_1^2,x_1^3\in u_1\cap X_1(S)$. Analogously, there exist $x_i^1,x_i^2,x_i^3\in u_i\cap X_1(S)$ for $i=2,3,4,5$. Let us consider $S'=\{v_1,v_2,v_3,v_4,v_5\}$, where
        \begin{align*}
            &v_1=(u_1\setminus\{x\})\cup \{x_5^1\},\\
            &v_2=(u_2\setminus\{x\})\cup \{x_1^1\},\\
            &v_3=(u_3\setminus\{x,x_3^3\})\cup \{x_1^2,x_2^1\},\\
            &v_4=(u_4\setminus\{x,x_4^2,x_4^3\})\cup \{x_1^3,x_2^2,x_3^1\},\\
            &v_5=(u_5\setminus\{x,x_5^2,x_5^3\})\cup \{x_2^3,x_3^2,x_4^1\}.
        \end{align*}
        It can be easily checked that for every $i\neq j$ we have $|u_i\cap u_j|=|v_i\cap v_j|$. It turns out that $S'$ is a $2$-packing of $\Kneser{n}{r}$ of size $5$ and $|X_3^\geq(S')|=|X_3^\geq(S)|-1$.

        In any case, if $X_3^\geq (S)\neq \emptyset$, we get a $2$-packing $S'$ of size $5$ with $|X_3^\geq (S')|<|X_3^\geq (S)|$. If $X_3^\geq (S')\neq \emptyset$, we repeat the procedure until we get a $2$-packing of five vertices for which each element occurs at most twice.
    \end{proof}
\end{lemma}

From Proposition~\ref{prp:3r-3}, and some other results of similar nature, one can suspect that $\packing(n,r)$ will be small if $n$ is close to $3r$. The following result describes how close $n$ has to be to $3r-3$ in order to have  $\packing(n,r)$ equal to $3$ or $4$.

\begin{theorem}\label{thm: packing 3 and 4}
    If $r$, $n$ and $t$ are positive integers such that $2\leq t\leq r-1$ and $n=3r-t$, then:
    \begin{enumerate}
        \item if $t\le \frac{r+5}{5}$, then $\packing(n,r)=3$;
        \item if $\frac{r+5}{5}<t \le \frac{2r+9}{9}$, then $\packing(n,r)=4$.
        %\item If $r < \frac{9}{2}(t-1)$, then $\packing(n,r)\geq 5$.
    \end{enumerate}
    \begin{proof}
    \begin{enumerate}
        \item Let us suppose $t\le \frac{r+5}{5}$, which is equivalent to $r\geq 5(t-1)$.
        Assume $\packing(n,r)\geq 4$ and let $S$ be a packing of $\Kneser{n}{r}$ of cardinality $4$. Note that since $r\geq 5(t-1)>3(t-1)$, we have $2\leq t\leq \frac{r+3}{3}$ and $r\geq 3$. Thus, by Lemma \ref{lem: packing of size 4}, we may consider a $2$-packing $S$ for which $i_x(S)\leq 2$ for every $x\in [n]$. Let us note that by Observation \ref{obs:pack} we have $|u\cap v|\leq t-1$ for each $u,v\in S$, and it follows
        $$|X_2(S)|=\sum_{{\substack{u,v\in S\\ u\neq v}}}\underbrace{|u\cap v|}_{\leq t-1}\leq \binom{4}{2}(t-1)=6(t-1).$$
        On the other hand, since $|X_1(S)|\leq n-|X_2(S)|$, we have
        \[
        \begin{array}{lcl}
        \displaystyle{ 4r=|S|r=\sum_{x\in [n]}i_x(S) } 
        & = & 2|X_2(S)|+|X_1(S)| \\
        & \le & n+|X_2(S)| \\
        & \le & n+6(t-1) \\
        & \le & 3r+5t-6.
        \end{array}
        \] 
        Thus, $r\leq 5(t-1)-1$ and we arise to a contradiction since $r\geq 5(t-1)$. Therefore $\packing(n,r)\leq 3$. In order to see that it is equal to $3$ it is enough to consider the set $S$ given by $S=\{u_1,u_2,u_3\}$ where
        $$u_1 = [r],\quad u_2 = [t-1]\cup[(r+1)..(2r-t+1)],\quad u_3 = \{1\}\cup[(2r-t+1)..(3r-t)].$$
        We have $|u_1\cap u_2|=t-1$ and $|u_1\cap u_3|=|u_2\cap u_3|=1$. So, $S$ is a 2-packing and $\packing(n,r)=3$.

        \item Now, let us suppose $\frac{r+5}{5}<t \le \frac{2r+9}{9}$, which is equivalent to $\frac{9}{2}(t-1)\leq r < 5(t-1)$.
        Assume $\packing(n,r)\geq 5$ and let $S$ be a packing of $\Kneser{n}{r}$ of size $5$. Note that since $r\geq \frac{9}{2}(t-1)>4(t-1)$, we have $2\leq t\leq \frac{r+4}{4}$ and $r\geq 3$. Thus, by Lemma \ref{lem: packing of size 5}, we may consider a $2$-packing $S$ for which $i_x(S)\leq 2$ for every $x\in [n]$. By Observation \ref{obs:pack} we have $|u\cap v|\leq t-1$ for each $u,v\in S$. Consequently,
        $$|X_2(S)|=\sum_{{\substack{u,v\in S\\ u\neq v}}}\underbrace{|u\cap v|}_{\leq t-1}\leq \binom{5}{2}(t-1)=10(t-1).$$
        It follows that
        \[
        \begin{array}{lcl}
        \displaystyle{ 5r =|S|r = \sum_{x\in [n]}i_x(S) } 
        & = & 2|X_2(S)|+|X_1(S)| \\
        & \le & n+|X_2(S)| \\
        & \le & n+10(t-1)=3r+9t-10.
        \end{array}
        \] 
        Thus, $r\leq \frac{9}{2}(t-1)-\frac{1}{2}$ and we arise to a contradiction since $r\geq \frac{9}{2}(t-1)$. Therefore $\packing(n,r)\leq 4$. In order to see that it is equal to $4$ it is enough to consider the set $S$ given by $S=\{u_1,u_2,u_3,u_4\}$ where
        \begin{align*}
            u_1 &= A_{12}\cup A_{13}\cup A_{14}\cup B_1\\
            u_2 &= A_{12}\cup A_{23}\cup A_{24}\cup B_2\\
            u_3 &= A_{13}\cup A_{23}\cup A_{34}\cup B_3\\
            u_4 &= A_{14}\cup A_{24}\cup A_{34}\cup B_4
        \end{align*}
        where the sets in $\big\{\{A_{ij}\}_{i\neq j}\cup\{B_i\}_{i=1}^{4}\big\}$ are pairwise disjoint with $|A_{ij}|=t-1$, and $|B_i|=r-3(t-1)$. Note that this is possible since 
        $$6(t-1)+4(r-3(t-1))=4r-6(t-1)=\underbrace{(3r-t)}_{=n}+\underbrace{(r-(5t-6))}_{\leq 0}\leq n.$$
        We have $|u_i\cap u_j|=t-1$ for every $i\neq j$. So, $S$ is a packing and $\packing(n,r)=4$.
    \end{enumerate}
    \end{proof}
\end{theorem}

\begin{corollary}
    If $\frac{14}{5}r-1\leq n\leq 3r-2$, then $\packing(n,r)=3$.
\end{corollary}

\begin{corollary}
    If $\frac{25}{9}r-1\leq n< \frac{14}{5}r-1$, then $\packing(n,r)=4$.
\end{corollary}

\section{Concluding remarks}

In Section~\ref{sec:monotone}, we proved that $\gamma_k$ is monotonically decreasing, where $r$ is fixed and $n\ge 2(k+r)$ is growing. We wonder if the lower bound on $n$ could be improved, and pose the following question.

\begin{problem}
Is there an integer $n_0$, where $n_0 < 2(k+r)$, such that for any $n \geq n_0$, we have $\kdom{k}(n,r) \geq \kdom{k}(n+1,r)$?
\end{problem}

In Section~\ref{sec:exact}, we established exact values of $\kdom{k}(n,r)$ and $\ktuplet{k}(n,r)$ for all $n$, where $n\ge r(k+r)-1$. It is thus natural to ask what are the values of these two invariants when $n$ is smaller than $r(k+r)-1$. In particular, the most interesting case seems to be that of odd graphs, that is, when $n=2r+1$. 

\begin{problem}
Determine or provide upper and lower bounds on $\kdom{k}(2r+1,r)$, $\ktuple{k}(2r+1,r)$ and $\ktuplet{k}(2r+1,r)$, for $r>2$.
\end{problem}

We mentioned an open problem in Section~\ref{sec:packing}, which is concerned with the exact values of $\packing(3r-3,r)$ where $r\le 8$ (recall that for $r\ge 9$, we already established the values of $\packing(3r-3,r)$).

\begin{problem}
 Is it true that the lower bounds in Table~\ref{tab:2pack} are in fact exact values of $\rho_2(3r-3,r)$, where $r\le 8$.
\end{problem}

Regarding the 2-packing number of odd graphs $K(2r+1,r)$, we have proved in Proposition \ref{prp: 2packingIncreasing} that  $\packing(2r'+1,r')\geq C \cdot \packing(2r+1,r)$ when $r'>r+1$, where $C=2$. In \cite{cornet2023k} it is shown that $\packing(7,3)=7$ and $\packing(11,5)=66$. So, we wonder whether this $C$ could be improved.
\begin{problem}
    Is there some $C(r)>2$ such that $\packing(2(r+2)+1,r+2)\geq C(r)\packing(2r+1,r)$ holds for all $r\ge 2$?
\end{problem}

\section*{Acknowledgments}

The first and the third author were supported by the Slovenian Research and Innovation agency (grants P1-0297, J1-2452, J1-3002, and J1-4008). The second author was partially supported by the Argentinian National Agency for the Promotion of Research, Technological Development and Innovation (grant PICT-2020-03032), the Argentinian National Council for Scientific and Technical Research (grant PIP CONICET 1900) and the National University of Rosario (grant PID 80020210300068UR). Research of the fourth author was supported in part by the South African National Research Foundation under grant number 132588 and the University of Johannesburg. The second and fourth authors thank the University of Maribor for their hospitality.

\bibliography{DomKneser}

\end{document}